\documentclass[11pt]{amsart}
\usepackage{latexsym, amsmath, amssymb, longtable, booktabs,amscd, color}

\usepackage{bbm}

\usepackage[english]{babel}
\usepackage[utf8]{inputenc}

\usepackage[all]{xy}
\usepackage[symbol]{footmisc}

\usepackage{tikz-cd}
\usepackage[hidelinks]{hyperref}

\numberwithin{equation}{section}
\newtheorem{thm}{Theorem}[section]
\newtheorem{theorem}[thm]{Theorem}
\newtheorem{cor}[thm]{Corollary}

\newtheorem{prop}[thm]{Proposition}
\newtheorem{lemma}[thm]{Lemma}
\theoremstyle{definition}
\newtheorem{definition}[thm]{Definition}
\theoremstyle{remark}
\newtheorem{rk}[thm]{Remark}

\theoremstyle{definition}

\theoremstyle{remark}

\theoremstyle{remark}

\newcommand{\h}{\mathcal{H}}
\newcommand{\z}{\mathbb{Z}_{\geq 0}}
\newcommand{\p}{\{0,1,...,p-1\}}
				
\title[]{Lucas' Theorem modulo $p^n$ via Harmonic sums}
\author{Suneel Kumar}

\address{Department of Mathematics, Indian Institute of Technology Kanpur, U.P. 208016, India}
\email{suneelm145@gmail.com}

\begin{document}

\maketitle

\begin{abstract}
Lucas' theorem states that for a prime $p$, integers $a,b\geq0$ and $0\leq r,s\leq p-1$, one has $\binom{pa+r}{pb+s}\equiv\binom{a}{b}\binom{r}{s}\bmod p$. Let $D(p^n)$ be the set of pairs $(r,s)$ with $s\leq r$ for which this congruence holds modulo $p^n$ for all $a,b\geq0$. Rowland characterized $D(p^2)$ in terms of harmonic sums and symmetries of Pascal's triangle. For $n<p$, we give a criterion for a point $(r,s)$ to lie in $D(p^n)$ in terms of generalized harmonic sums of order at most $n-1$. We show that $D(p^n)$ is invariant under the reflection $(r,s)\mapsto(r,r-s)$ and the rotation $(r,s)\mapsto(p-1-s,r-s)$, so that with every point it contains its full orbit under the dihedral group $D_3$. Applied to the vertices, our criterion gives a result of Jacobsthal--Ljunggren type modulo $p^n$. In particular, we recover Zhao's result $(0,0)\in D(p^4)$ if and only if $p$ is a Wolstenholme prime. We also study the centroid and the edge midpoints, whose membership in $D(p^3)$ and $D(p^4)$ is governed by Bernoulli numbers and polynomials, Wolstenholme primes and super Wieferich primes. In particular, for $p=205129$ the centroid lies in $D(p^3)$, answering a question of Rowland. 
\end{abstract}

\section{Introduction}\label{intro}
Lucas' theorem is a classical tool for computing binomial coefficients modulo a prime. In $1878$, Lucas \cite{lucas} proved that for any prime $p$ and $r, s\in\p$,
$${pa+r\choose pb+s}\equiv {a\choose b}{r\choose s}\bmod p\quad\text{for all integers }a,b\geq 0.$$
In general, the above congruence fails modulo higher powers of $p$, and a natural question is for which $(r,s)$ it still holds. We denote by $D(p^n)$ the set of $(r,s)$ with $r, s\in\p$ such that
$${pa+r\choose pb+s}\equiv {a\choose b}{r\choose s}\bmod p^n\quad\text{for all integers }a,b\geq 0.$$
By Lucas' theorem, $(r,s)\in D(p)$ for all $r,s\in \p$. Rowland \cite{rowland} gave an explicit criterion for $(r,s)$ to lie in $D(p^2)$ in terms of harmonic sums and in terms of symmetries of Pascal's triangle (Rowland denotes the set $D(p^2)$ by $D(p)$). In this article, we give a criterion for a point $(r,s)$ to lie in $D(p^n)$ in terms of generalized harmonic sums of order at most $n-1$. We also study points such as the vertices, the centroid and the edge midpoints, as well as the size of $D(p^n)$, and in particular we answer a question of Rowland \cite{rowland} about $D(p^3)$. Throughout, we assume that $p\geq 5$ and $r,s\in \p$ with $r\geq s$, and do not repeat these assumptions in each statement. For $n\geq2$, the assumption $r\geq s$ is not a restriction, since no pair with $s>r$ lies in $D(p^n)$. This can be seen by taking $a=1$ and $b=0$: the right-hand side is ${1\choose 0}{r\choose s}=0$ as $s>r$, whereas ${p+r\choose s}\not\equiv 0\bmod p^2$,
because the factor $p$ occurs exactly once in $(p+r)!$ and not at all in $s!\,(p+r-s)!$, as $s,\,p+r-s<p$.

Extensive work has been carried out on Lucas' theorem and its generalizations. We refer to {Me\v strovi\'c} \cite{Romeo} for an excellent survey. Binomial coefficients modulo prime powers have been of interest since the early nineteenth century. Even before Lucas' theorem, Babbage \cite{babbage} proved in $1819$ that for every prime $p\geq 3$
\begin{align}\label{babbage_con}
    {2p-1\choose p-1}\equiv 1\bmod p^2.
\end{align}
In $1862$, Wolstenholme \cite{Wolstenholme} proved that for every prime $p\geq 5$ the above congruence also holds modulo $p^3$. In $1949$ Jacobsthal and Ljunggren \cite{Lj} proved the following congruence
\begin{align}\label{eqn Jl type intro}
    {pa\choose pb}\equiv {a\choose b}\bmod p^3\quad\text{for all } a, b\geq 0\quad\text{and }p\geq 5.
\end{align}
This congruence shows that $(0,0)\in D(p^3)$ for $p\geq 5$. Several authors have studied analogues of Lucas' theorem modulo higher powers of $p$. In $1990$, Bailey \cite{Bailey} proved that for $r,s\in\p$,
$${p^2a+r\choose p^2b+s}\equiv {a\choose b}{r\choose s}\bmod p^2,$$
and that the same congruence holds with $p^2$ replaced by $p^3$ when $p\geq 5$. In the same year, Davis and Webb \cite{davis-webb} obtained a Lucas-type theorem for prime powers, and independently Granville \cite{Granville} obtained a similar result. Davis and Webb \cite{Davis-Webb1993} extended Bailey's congruence to every modulus $p^k$ by showing that
$${p^ka+r\choose p^kb+s}\equiv {p^{\lfloor (k-1)/3\rfloor}a\choose p^{\lfloor (k-1)/3\rfloor}b}{r\choose s}\bmod p^k\quad\text{for } p\geq 5 \text{ and } k\geq 1.$$
Zhao \cite[Theorem 3.5]{Zhao} in $2007$ extended Bailey's congruence modulo $p^5$ for $p\geq 7$. In the same paper \cite[Theorem 3.2]{Zhao}, Zhao also extended the Jacobsthal-Ljunggren congruence and proved that
$${pa\choose pb}\equiv {a\choose b}\bmod p^4\quad\text{for all } a,b\geq 0$$
if and only if $p$ is a Wolstenholme prime, that is, a prime $p$ for which \eqref{babbage_con} holds modulo $p^4$. In our notation, this says that $(0,0)\in D(p^4)$ if and only if $p$ is a Wolstenholme prime.

Next, we define the notations required to state our main results. Let $\mathbb{Z}_p$ denote the ring of $p$-adic integers. For $m\in\{1,\ldots,p-1\}$, we define
$$\alpha(x,m) =\prod_{i = 1}^{m}\left(1 + \frac{px}{i}\right)\in\mathbb{Z}_p[x],$$
and by convention, we set $\alpha(x,0) = 1$. For $f,g\in\mathbb{Z}_p[x]$, we write $f\equiv g\bmod p^n$ if all coefficients of $f-g$ are divisible by $p^n$. With these notations, we have the following criterion. 
\begin{theorem}\label{alpha_crit_n}
    Let $r,s\in\p$ with $r\geq s$ and assume $n < p$. Then $(r,s)\in D(p^n)$ if and only if
    \begin{enumerate}
        \item $\alpha(x,r)\equiv \alpha(x, s)\equiv \alpha(x,r-s)\bmod p^n$; and
        \item $\alpha(x,p-1)\,\alpha(x+1,r)\equiv \alpha(1,r)\,\alpha(x,r)\bmod p^n$.
    \end{enumerate}
\end{theorem}

As a corollary of the above theorem, we are also able to deduce the above criterion entirely in terms of generalized harmonic sums. To do this, we introduce a few more notations. For $m\in\{0,1,\ldots,p-1\}$ and an integer $l\geq 1$, we define  
    $$\h_m^{(l)} = \sum_{1\leq k\leq m}\frac{1}{k^l},$$
by convention, we set $\h_0^{(l)} = 0$ for all $l$. We have the following criterion. 
\begin{cor}\label{crit_power}
Let $r,s\in\p$ with $r\geq s$ and assume $n < p$. Then $(r,s)\in D(p^n)$ if and only if
\begin{enumerate}
    \item[(i)] $\h_r^{(l)}\equiv \h_s^{(l)}\equiv \h_{r-s}^{(l)}\bmod p^{n-l}$ for all $1\leq l\leq n-1$; and
    \item[(ii)] $\displaystyle\h^{(i)}_{p-1}\equiv\sum_{t=1}^{n-i-1}(-1)^{t+1}{i+t-1\choose t}\,p^t\,\h^{(i+t)}_r\bmod p^{n-i}$
    for all $1\leq i\leq n-2$.
\end{enumerate}
\end{cor}

We remark that the hypothesis $n< p$ in Theorem \ref{alpha_crit_n} is used only in the proof of the forward direction, whereas in the proof of Corollary \ref{crit_power} it is used in both directions. We also remark that when $n = 2$, condition $(ii)$ in Corollary \ref{crit_power} is empty. Hence, it recovers Rowland's criterion \cite{rowland}. 

Pascal's triangle is invariant under one reflection symmetry, namely $\binom rs=\binom r{r-s}$. Further, by fixing the number of rows of Pascal's triangle to be $p$, from $0$ to $p-1$, then modulo $p$, up to sign, it is invariant under all the symmetries of the dihedral group $D_3$ (see Definition \ref{dfn symmetric}). We prove in Proposition \ref{symmetry_D_p^n} that the points of $D(p^n)$ satisfy the rotation symmetries in the sense of Definition \ref{dfn symmetric}. For $n=2$, Rowland's criterion \cite{rowland} shows that converse also holds, that is a point satisfying these symmetries lies in $D(p^2)$. However, this fact does not hold in general (see Remark \ref{rk centroid}).  Further the following result shows that $D(p^n)$ is invariant under the rotation $(r,s)\mapsto(p-1-s,r-s)$ and the reflection $(r,s)\mapsto(r,r-s)$. 
\begin{theorem}\label{full orbit}
   Let $1\leq n<p$. If $(r,s)\in D(p^n)$, then $(r,r-s)\in D(p^n)$ and $(p-1-s,\,r-s)\in D(p^n)$.
\end{theorem}

As a direct corollary of the above theorem we have the following result, which shows that if $(r,s)\in D(p^n)$ then its full orbit under the action of the dihedral group $D_3$ lies in $D(p^n)$. 
\begin{cor}\label{cor_orbit}
  Let $1\leq n<p$. If $(r,s)\in D(p^n)$, then  $$\left\{\begin{array}{lll}
   (r,s), & (p-1-s,\,r-s), & (p-1-r+s,\,p-1-r),\\
   (r,\,r-s), & (p-1-r+s,\,s), & (p-1-s,\,p-1-r)
 \end{array}\right\}\subseteq D(p^n).$$
\end{cor}
Corollary \ref{cor_orbit} together with the Jacobsthal-Ljunggren congruence in \eqref{eqn Jl type intro} implies that the set of vertices $\{(0,0), (p-1,0),(p-1,p-1)\}$ is contained in $D(p^3)$ for all $p\geq 5$. 

Next we study the special points of Pascal's triangle with $p$ rows from $0$ to $p-1$, such as the vertices, the centroid and the edge mid points. As noted above, the Jacobsthal-Ljunggren congruence in \eqref{eqn Jl type intro} is the same as $(0,0)\in D(p^3)$. The following result generalizes the Jacobsthal-Ljunggren congruence. 
\begin{cor}\label{main_base_gen}
Let $n$ be a positive integer with $n < p$. Then the following are equivalent 
\begin{enumerate}
    \item the vertices lie in $D(p^n)$;
    \item $(0,0)\in D(p^n)$;
    \item  $ \h_{p-1}^{(l)}\equiv 0\bmod p^{n-l}$ for all $l\in \{1,2,\cdots, n-2\}$.
\end{enumerate}
\end{cor}
As a special case of the above corollary, we recover the Jacobsthal-Ljunggren congruence in \eqref{eqn Jl type intro} in Corollary \ref{vertex_small_n} $(a)$ and Zhao's result \cite{Zhao} for $n =4$ in Corollary \ref{vertex_small_n} $(b)$ and the result in Corollary \ref{vertex_small_n} $(c)$ can be deduced from Zhao \cite[Theorem 3.2]{Zhao}. 

McIntosh \cite{McIntosh} conjectured that no prime satisfies ${2p-1\choose p-1}\equiv1\bmod p^5$. By Corollary \ref{vertex_small_n}(c) this conjecture is equivalent to $(0,0)\notin D(p^5)$ for all prime $p\geq 7$. Note that a prime $p$ with $(0,0)\in D(p^5)$ satisfies ${2p-1\choose p-1}\equiv1\bmod p^5$, and in particular is a Wolstenholme prime. However, the only Wolstenholme primes below $10^{11}$ are $16843$ and $2124679$ \cite{BHMT} neither of which satisfies ${2p-1\choose p-1}\equiv1\bmod p^5$. Thus $(0,0)\notin D(p^5)$ for all primes $p<10^{11}$. Note that $D(p^n)\subseteq D(p^5)$ for all $n\geq 5$. Therefore, $(0,0)\notin D(p^n)$ for all $n\geq 5$ and for all primes $p<10^{11}$.

When $p\equiv 1\bmod 3$, the centroid $(2(p-1)/3,(p-1)/3)$ exists. As noted above, the centroid  in general does not belong $D(p^n)$ even though it satisfies symmetries in the sense of Definition \ref{dfn symmetric}. For $n= 3,4$, we have the following result for centroid and we refer to Theorem \ref{centroid_n} for arbitrary $n$. 

\begin{cor}\label{centroid_small_n}
    Let $p\equiv1\bmod3$ and $m=\frac{p-1}{3}$. Then 
\begin{enumerate}
    \item $(2m,m)\in D(p^3)$ if and only if $B_{p-2}(1/3)\equiv0\bmod p$.
    \item $(2m,m)\in D(p^4)$ if and only if $p$ is a Wolstenholme prime and $\h^{(2)}_m\equiv0\bmod p^2$.
\end{enumerate}
\end{cor}
Here, $B_k(x)$ denotes the $k$-th Bernoulli polynomial.  Z.-W.~Sun \cite{sun2010} observed that $205129$ is the first prime satisfying $B_{p-2}(1/3)\equiv 0 \bmod p$, and that the second such prime must be greater than $2\times 10^{7}$. He further remarks that standard heuristic
arguments suggest there are infinitely many such primes. Therefore, the only known prime for which the centroid belongs $D(p^3)$ is $205129$. In the context of Corollary \ref{centroid_small_n} $(2)$, we note that both known Wolstenholme primes, $16843$ and $2124679$, satisfy $p\equiv1\bmod3$, but for both $\h^{(2)}_{(p-1)/3}\not\equiv0\bmod p$. Hence there is no known prime for which the centroid lies in $D(p^4)$.

 A prime $p$ is called a \emph{Wieferich prime} if $2^{p-1}\equiv 1\bmod p^2$, and a \emph{super Wieferich prime} if $2^{p-1}\equiv 1\bmod p^3$. The only known Wieferich primes are $1093$ and $3511$, and these are the only ones up to $4\times10^{17}$ (see Murty and S\'eguin \cite[p.~2]{murty}). Neither of them is a super Wieferich prime, and heuristically there should be only finitely many super Wieferich primes \cite{murty}. For $n =3$, we have the following result for edge midpoints and refer to Theorem \ref{edge-midpoints} for arbitrary $n$. 
\begin{cor}\label{edge midpoint_p^3}
Let $e=(p-1)/2$. Then the edge midpoints lie in $D(p^3)$ if and only if $\h_e^{(1)}\equiv 0\bmod p^2 $ if and only if $p$ is a super Wieferich prime.
\end{cor} 
Since no super Wieferich prime is known, no prime is known for which the edge midpoints lie in $D(p^3)$. 

In the context of McIntosh conjecture \cite{McIntosh}, it is natural to ask whether $D(p^n)$ being non-empty implies that $(0,0)\in D(p^n)$. In Theorem \ref{edge-midpoints}, we prove that if the edge midpoints lie in $D(p^n)$ then $(0,0)\in D(p^n)$. However, in general, this question appears to be more involved and remains open.

Our article is organised as follows. In Section \S \ref{harmonic sum sec}, we prove some technical lemmas. Theorem \ref{alpha_crit_n} and Corollary \ref{crit_power} are proved in Section \S \ref{main criterian}. Theorem \ref{full orbit} and Corollary \ref{cor_orbit} are established in Section \S \ref{symmetry_sec}. The Jacobsthal–Ljunggren type congruence, including Corollary \ref{main_base_gen}, are proved in Section \S \ref{vitices sec}. Finally, results concerning centroids and edge midpoints, including Corollary \ref{centroid_small_n} and Corollary \ref{edge midpoint_p^3}, are proved in Sections \S \ref{centroid sec} and \S \ref{edge midpoint sec}, respectively.


\section{Notation and known results}
We introduce some notation which will be used throughout the paper. For $m,k\in\p$, we define
$$\h_m(k) := \sum_{1\leq i_1<i_2<\cdots<i_k\leq m}\frac{1}{i_1i_2\cdots i_k}\quad\text{for } 1\leq k\leq m,$$and by convention, we set $\h_m(0) = 1$, and $\h_m(k) = 0$ for $k> m$. We note that $\h_m(k)\in\mathbb Z_p$ for all $m\leq p-1$, so it is well defined modulo $p^n$ for all $n\geq 1$. Using the definition of $\alpha(x,m)$, observe that $$\alpha(x,m) = \underset{0\leq k\leq m}\sum (px)^k\h_m(k)\in\mathbb{Z}_p[x].$$
Note that $\alpha(x,m)\equiv 1\bmod p$ for all $m\leq p-1$. Therefore, for every $a\in \mathbb{Z}$, $\alpha(a,m)$ is a unit element in $\mathbb{Z}_p$ and hence its image is a unit in the ring $\mathbb{Z}/p^n\mathbb{Z}$.
Next, we recall some known results which will be needed later.

\begin{theorem}[Wolstenholme, \cite{Wolstenholme}]\label{Wolstenholme}
Let $p\geq 5$ be a prime. Then $\h_{p-1}^{(1)}\equiv 0\bmod p^2$ and $\h_{p-1}^{(2)}\equiv0\bmod p$.  Moreover,     
 ${{2p-1}\choose p-1}\equiv 1\bmod p^3.$
 \end{theorem}
 \begin{rk}\label{rk alpha(1,p-1)}
      We remark that for any prime $p$, ${{2p-1}\choose p-1} = \alpha(1,p-1)$, which follows from the observation that $(2p-1)! = p!(p-1)!\alpha(1,p-1)$.
  \end{rk}
We recall that, for each non-negative integer $n$, the $n$-th Bernoulli polynomial is defined as follows
$$B_n(x) = \sum_{0\leq k\leq n}\binom{n}{k}B_{n-k}x^k,$$
where $B_{n-k}$ denotes the $(n-k)$-th Bernoulli number. One of the important properties of the Bernoulli polynomials is the symmetry $B_n(1-x) = (-1)^nB_n(x)$, and $B_n = 0$ for all odd integers $n\geq 3$.
\begin{theorem}[\cite{Kenneth-Rosen} - Ch. $15$, \S $1$, Theorem $1$]\label{bernoulli poly}
For all positive integers $n$ and $m$, 
\begin{align*}
    \sum_{1\leq k\leq n-1}k^m = \frac{1}{m+1}(B_{m+1}(n)- B_{m+1}).
\end{align*}    
\end{theorem}
\begin{theorem}[\cite{Kenneth-Rosen} - Ch. $15$, \S $1$, Theorem $3$]\label{berno p-adic int}
    For any positive integer $m$, $B_{2m}$ is a $p$-adic integer if $(p-1)\nmid 2m$. 
\end{theorem}
\begin{theorem}[Lehmer \cite{lehmer}]\label{Lehmer}
Let $p\geq 5$ be a prime and let $q_p(2)=\dfrac{2^{p-1}-1}{p}$ denote the Fermat quotient of $p$ to base $2$. Then
\[
    \mathcal{H}_{(p-1)/2}(1)\;\equiv\;-2\,q_p(2)+p\,q_p(2)^{2} \pmod{p^{2}} .
\]
\end{theorem}


\section{Properties of $\alpha(x,m)$ and Harmonic sums}\label{harmonic sum sec}

\begin{lemma}\label{alpha_factor}
For $m\in\p$, we have
\begin{enumerate}
  \item $\alpha(-1-x,p-1)=\alpha(x,p-1)$;
  \item  $\alpha(x,p-1-m)\,\alpha(-1-x,m)=\alpha(x,p-1)\,\alpha(-1,m)$.
\end{enumerate}
\end{lemma}

\begin{proof}
$(1)$. Since $i\mapsto p-i$ permutes $\{1,\dots,p-1\}$ we get
$$\alpha(-1-x,p-1)=\prod_{i=1}^{p-1}\frac{i-p-px}{i}=\prod_{i=1}^{p-1}\frac{-(i+px)}{p-i}
=(-1)^{p-1}\prod_{i=1}^{p-1}\frac{i+px}{i}=\alpha(x,p-1).$$
$(2)$, we note that 
$$\alpha(x,p-1)=\alpha(x,p-1-m)\prod_{i=p-m}^{p-1}\left(1+\frac{px}{i}\right).$$
Writing $i=p-j$ with $1\leq j\leq m$, we have
$$1+\frac{px}{p-j}=\frac{p(1+x)-j}{p-j}=\frac{1-\frac{p(1+x)}{j}}{1-\frac{p}{j}},$$
so the last product equals $\alpha(-1-x,m)/\alpha(-1,m)$. Multiplying by $\alpha(-1,m)$ gives
$(2)$.
\end{proof}

\begin{lemma}\label{zero_phi_one}
Let $\Phi\in(\mathbb{Z}/p^N\mathbb{Z})[x]$ with $\deg\Phi\leq d\leq p-1$. If $\Phi(m)=0$ for
all integers $m\geq0$ lying in at least $d+1$ distinct residue classes modulo $p$, then
$\Phi=0$.
\end{lemma}

\begin{proof}
We prove the result by induction on $N$. Reducing modulo $p$, $\Phi$ becomes a polynomial over $\mathbb Z/p\mathbb Z$ of degree at most $d$ with at least $d+1$ roots, hence zero. If $N=1$, we are done. Otherwise, every coefficient of $\Phi$ is divisible by $p$, so $\Phi=p\Psi$ for some $\Psi\in(\mathbb{Z}/p^{N-1}\mathbb{Z})[x]$ with $\deg\Psi\leq d$, and $\Psi(m)=0$ in $\mathbb{Z}/p^{N-1}\mathbb{Z}$ for the same integers $m$. By induction
$\Psi=0$, hence $\Phi=0$.
\end{proof}

\begin{lemma}\label{alpha_odd}
    Let $n < p$ and $r\in\p$. Assume $\alpha(x,p-1)\alpha(x+1,r)\equiv \alpha(1,r)\alpha(x,r)\bmod p^n$. Then
    $$\alpha(x,r)\alpha(-x, r)\equiv 1\bmod p^n.$$
\end{lemma}
\begin{proof}
  By hypothesis, we have
  \begin{align}\label{eqn 1 alpha odd}
    \alpha(x,p-1)\alpha(x+1,r)\equiv \alpha(1,r)\alpha(x,r)\bmod p^n.
\end{align}
Substituting $x\mapsto-1-x$ in the above equation and using $\alpha(-1-x,p-1)=\alpha(x,p-1)$ from
Lemma \ref{alpha_factor}, we get
$$\alpha(x,p-1)\,\alpha(-x,r)\equiv\alpha(1,r)\,\alpha(-1-x,r)\bmod p^n .$$
We multiply \eqref{eqn 1 alpha odd} by $\alpha(-x,r)$ and the above congruence by $\alpha(x+1,r)$ and subtract,we get
$$\alpha(x,r)\,\alpha(-x,r)\equiv\alpha(x+1,r)\,\alpha(-(x+1),r)\bmod p^n .$$
This implies that $\alpha(a,r)\alpha(-a,r)\equiv\alpha(a+1,r)\alpha(-(a+1),r)\bmod p^n$ for all integers
$a\geq0$ which further implies that $\alpha(a,r)\alpha(-a,r)\equiv\alpha(0,r)^2 \equiv 1\bmod p^n$. Since the coefficients of $x^k$ in $\alpha(x,r)$ and $\alpha(-x,r)$ are divisible by
$p^k$, the reduction modulo $p^n$ of $\alpha(x,r)\alpha(-x,r)-1$ has degree at most
$n-1<p$. Hence, by Lemma \ref{zero_phi_one} we get the desired result $\alpha(x,r)\,\alpha(-x,r)\equiv1\bmod p^n$.
\end{proof}

\begin{lemma}\label{shift_power}
Let $r\in\p$ and $1\leq n<p$. Then the following are equivalent
\begin{enumerate}
    \item $\alpha(x,p-1)\,\alpha(x+1,r)\equiv\alpha(1,r)\,\alpha(x,r)\bmod p^n$;\vspace{.5em}
    \item $\h^{(i)}_{p-1}\equiv\sum_{t=1}^{n-i-1}(-1)^{t+1}{i+t-1\choose t}\,p^t\,\h^{(i+t)}_r\bmod p^{n-i}
\quad\text{for all } 1\leq i\leq n-2.$
\end{enumerate}
\end{lemma}

\begin{proof}
We first note that the reduction modulo $p^n$ of
$\alpha(x,p-1)\alpha(x+1,r)-\alpha(1,r)\alpha(x,r)$ has degree at most $n-1<p$, since the
coefficients of $x^k$ in all these polynomials are divisible by $p^k$. Hence, by
Lemma \ref{zero_phi_one}, the statement $(1)$ holds if and only if
$$\alpha(a,p-1)\,\alpha(a+1,r)\equiv \alpha(1,r)\,\alpha(a,r)\bmod p^n\quad\text{for all integers } a\geq0 .$$
Since $\alpha(a, m)\in 1 + p\mathbb Z_p$, the $p$-adic logarithm $\log(\alpha(a,m))$ is
defined. Moreover, $\log$ is multiplicative, and $u\equiv v\bmod p^n$ if and only if
$\log u\equiv\log v\bmod p^n$ for $u,v\in1+p\mathbb{Z}_p$ (see \cite[Ch.II,\S $5$]{Neukirch}).
Therefore, $(1)$ holds if and only if
\begin{equation}\label{log_eq}
\log (\alpha(a,p-1)) + \log(\alpha(a+1,r)) \equiv \log(\alpha(1,r)) +\log(\alpha(a,r))\bmod p^n
\end{equation}
for all integers $a\geq0$. Since $\alpha(a,m)$ is a product, we have
$$\log(\alpha(a,m))=\sum_{i=1}^{m}\log\left(1+\frac{pa}{i}\right)=\sum_{l\geq1}e_l(m)\,a^l,
\qquad\text{where } e_l(m)=\frac{(-1)^{l-1}p^l\,\h^{(l)}_m}{l}.$$
Since $n<p$, we have $v_p(e_l(m))\geq l-v_p(l)\geq n$ for $l\geq n$. Hence, putting
$L_m(x)=\sum_{l=1}^{n-1}e_l(m)\,x^l\in\mathbb{Z}_p[x]$, we get
$\log(\alpha(a,m))\equiv L_m(a)\bmod p^n$ for all integers $a\geq0$. Let
$\Psi(x)=L_{p-1}(x)+L_r(x+1)-L_r(1)-L_r(x)$. Therefore, \eqref{log_eq} is equivalent to
$\Psi(a)\equiv0\bmod p^n$ for all integers $a\geq0$. Since $\deg\Psi\leq n-1<p$, by
Lemma \ref{zero_phi_one} this holds if and only if $\Psi\equiv0\bmod p^n$. It therefore
suffices to show that $\Psi\equiv0\bmod p^n$ is equivalent to $(2)$.

Next, we compute the coefficients of $\Psi$. We note that
$$L_r(x+1)=\sum_{l=1}^{n-1}e_l(r)(x+1)^l=\sum_{i=0}^{n-1}\Bigl(\sum_{l=i}^{n-1}{l\choose i}e_l(r)\Bigr)x^i .$$
Since $L_{p-1}$ and $L_r$ have no constant term and the constant term of $L_r(x+1)$ is
$L_r(1)$, the constant term of $\Psi$ is zero. For $1\leq i\leq n-1$, the coefficient of
$x^i$ in $\Psi$ is $e_i(p-1)+\sum_{l=i+1}^{n-1}{l\choose i}e_l(r)$. Using
${l\choose i}/l={l-1\choose i-1}/i$ and multiplying by the unit $(-1)^{i-1}i$, we see that
all the coefficients of $\Psi$ are zero modulo $p^n$ if and only if
\begin{align*}
    e_i(p-1)&\equiv-\sum_{l=i+1}^{n-1}{l\choose i}e_l(r)\bmod p^n\quad\text{for all } 1\leq i\leq n-1 \\
    \iff p^i\h_{p-1}^{(i)} & \equiv -\sum_{l=i+1}^{n-1}(-1)^{l-i}{l-1\choose i-1}p^{l}\h_r^{(l)}\bmod p^n\quad\text{for all } 1\leq i\leq n-1 .
\end{align*}
Next, writing $l=i+t$ and dividing by $p^i$, we get
$$\h_{p-1}^{(i)} \equiv \sum_{t=1}^{n-1-i}(-1)^{t+1}{i+t-1\choose t}p^{t}\h_r^{(i+t)}\bmod p^{n-i}\quad\text{for all } 1\leq i\leq n-1.$$
Finally, when $i=n-1$ the sum is empty and the congruence reduces to
$\h^{(n-1)}_{p-1}\equiv0\bmod p$, which holds by Lemma \ref{valuation}. Hence only
$1\leq i\leq n-2$ need be imposed, which is $(2)$. This completes the proof.
\end{proof}

\begin{lemma}\label{valuation}
For any integer $1\leq k \leq p-2$, $\h^{(k)}_{p-1}\equiv0\ \bmod p$. Moreover, $\h^{(k)}_{p-1}\equiv0\ \bmod p^2$ when $k$ is odd and $k\leq p-3$.
\end{lemma}
\begin{proof}
Observe that for all $1\leq k\leq p-2$, the polynomial $X^k-1$ has at most $k$ roots in $\mathbb{Z}/p\mathbb{Z}$ implies that there is an integer $c\in\mathbb{Z}$ such that $p\nmid c$ and $c^k\not\equiv 1\bmod p$. Note that $i\mapsto ci$ gives bijection from $(\mathbb{Z}/p\mathbb{Z})^\times$ to $(\mathbb{Z}/p\mathbb{Z})^\times$ implies that 
$$\h^{(k)}_{p-1}\equiv\sum_{i=1}^{p-1}(ci)^{-k}=c^{-k}\,\h^{(k)}_{p-1}\bmod p,$$
that is, $(1-c^{-k})\,\h^{(k)}_{p-1}\equiv0\bmod p$. Since $1-c^{-k}$ is a unit, we get
$\h^{(k)}_{p-1}\equiv0\bmod p$ for all $1\leq k\leq p-2$. Next we prove moreover part $\h^{(k)}_{p-1}\equiv0\ \bmod p^2$ when $k$ is odd. Since $v_p\bigl(\frac pi\bigr) = 1$ for $1\leq i\leq p-1$, the expansion
$\bigl(1-\frac pi\bigr)^{-k}=\sum_{t\geq0}{k+t-1\choose t}\frac{p^t}{i^t}$ converges in $\mathbb Z_p$, and hence
$(p-i)^{-k}=-i^{-k}\bigl(1-\frac pi\bigr)^{-k}=-\sum_{t\geq0}{k+t-1\choose t}\frac{p^t}{i^{k+t}}$ as $k$ is odd. Using this and pairing $i$ with $p-i$, we get
\begin{align}\label{eqn pair i_p-i}
    2\,\h^{(k)}_{p-1} =-\sum_{t\geq1}{k+t-1\choose t}p^t\,\h^{(k+t)}_{p-1}.
\end{align}

Modulo $p^2$, only the term $t=1$ remains, so $2\,\h^{(k)}_{p-1}\equiv-k\,p\,\h^{(k+1)}_{p-1}\bmod p^2$. Since $k+1\leq p-2$, the first part gives $\h^{(k+1)}_{p-1}\equiv0\bmod p$. Hence $\h^{(k)}_{p-1}\equiv0\bmod p^2$.
\end{proof}

\begin{lemma}\label{depth_power}
Let $j,j'\in\p$ and $1\leq n<p$. Then the following are equivalent
\begin{enumerate}
    \item $\h_j(k)\equiv\h_{j'}(k)\bmod p^{n-k}\ \text{ for all }1\leq k\leq n-1$;
    \item  $\h^{(k)}_j\equiv\h^{(k)}_{j'}\bmod p^{n-k}\ \text{ for all }1\leq k\leq n-1 .$
\end{enumerate}
\end{lemma}

\begin{proof}
Let $d_k=\h_j(k)-\h_{j'}(k)$ and $\delta_k=\h^{(k)}_j-\h^{(k)}_{j'}$. By Newton's identity, we have $k\;\h_j(k)=\sum_{l=1}^{k}(-1)^{l-1}\h_j(k-l)\;\h^{(l)}_j$, and similarly for $j'$. Subtracting these two identities, we get
\begin{equation}\label{newton_diff}
k\,d_k=\sum_{l=1}^{k}(-1)^{l-1}\Bigl(\h_j(k-l)\,\delta_l+d_{k-l}\,\h^{(l)}_{j'}\Bigr).
\end{equation}
Suppose first that $\delta_l\equiv0\bmod p^{n-l}$ for all $1\leq l\leq n-1$. Since $\h_j(1) = \h_j^{(1)}$, we have $d_1\equiv 0\bmod p^{n-1}$. By induction, assume $d_i\equiv0\bmod p^{n-i}$ for all $i\leq k-1$. Note that this implies $d_i\equiv0\bmod p^{n-k}$ for all $i\leq k-1$, as $n-k \leq n-i$. Similarly,
$\delta_l\equiv0\bmod p^{n-k}$ for all $1\leq l\leq k$. Since also $d_0=0$, every term on the
right side of  \eqref{newton_diff} vanishes modulo $p^{n-k}$ implies that
$d_k\equiv0\bmod p^{n-k}$, as $k<p$ is a unit.

Conversely, suppose that $d_k\equiv0\bmod p^{n-k}$ for all $1\leq k\leq n-1$. Since
$\delta_1=d_1$, we have $\delta_1\equiv 0\bmod p^{n-1}$. By induction, assume
$\delta_i\equiv0\bmod p^{n-i}$ for all $i\leq k-1$. In \eqref{newton_diff}, the term with
$l=k$ is $(-1)^{k-1}\delta_k$, since $\h_j(0)=1$ and $d_0=0$. All the other terms, together
with $k\,d_k$, involve only $d_i$ with $i\leq k$ and $\delta_i$ with $i\leq k-1$, and hence
vanish modulo $p^{n-k}$. Therefore, \eqref{newton_diff} gives
$\delta_k\equiv 0\bmod p^{n-k}$.
\end{proof}

\section{Criterion for $D(p^n)$ via harmonic sums}\label{main criterian}

\begin{prop}\label{without mod p}
    Let $r,s\in\p$ with $r\geq s$. Then for any $a,b\in \z$
    \begin{align*}
        {{pa+r}\choose pb+s} &= {{pa}\choose pb}{{r}\choose s}\frac{\alpha(a,r)}{\alpha(b,s)\alpha(a-b,r-s)}\\
        \text{and}\quad{pa\choose pb} & = {a\choose b}\frac{\prod_{j = 0}^{a-1}\alpha(j,p-1)}{\prod_{j = 0}^{b-1}\alpha(j,p-1)\; \prod_{j = 0}^{a-b-1}\alpha(j,p-1)}.
    \end{align*}
\end{prop}
\begin{proof}
     We first observe that if $a< b$, then $pa+r < pb +s$ as $r < p$, which gives 
    $${pa + r\choose pb+ s} = 0 = {pa\choose pb}= {a\choose b}.$$
    Therefore, it remains to prove our result when $a\geq b$. Next, we note that 
    $${pa+r\choose pb+s} = \frac{(pa)!\prod_{1\leq i_1\leq r}(pa+i_1)}{(pb)!(p(a-b))!\prod_{1\leq i_2\leq s}(pb+i_2)\prod_{1\leq i_3\leq r-s}(p(a-b)+i_3)}.$$
    Further, we have $\prod_{i = 1}^y(px+i) = y!\ \alpha(x,y)$ for every $x\in\z$ and $y\in \p$. Substituting these into the above equation gives 
    $${{pa+r}\choose pb+s} = {{pa}\choose pb}{{r}\choose s}\frac{\alpha(a,r)}{\alpha(b,s)\alpha(a-b,r-s)}\quad\forall\ a,b\in\z.$$
    Next, observe that $(pm)! = p^m\; m!\; ((p-1)!)^m\; \prod_{j = 0}^{m-1}\alpha(j,p-1)$. Using this, we obtain
    \begin{align*}
        {pa\choose pb} &= \frac{p^a\; a!\; ((p-1)!)^a\prod_{j = 0}^{a-1}\alpha(j,p-1)}{p^b\; b!\; ((p-1)!)^b\prod_{j = 0}^{b-1}\alpha(j,p-1)\; p^{a-b}\; (a-b)!\; ((p-1)!)^{a-b}\prod_{j = 0}^{a-b-1}\alpha(j,p-1)}\\
        &= {a\choose b}\frac{\prod_{j = 0}^{a-1}\alpha(j,p-1)}{\prod_{j = 0}^{b-1}\alpha(j,p-1)\; \prod_{j = 0}^{a-b-1}\alpha(j,p-1)}.
    \end{align*}
\end{proof}

\begin{proof}[Proof of Theorem \ref{alpha_crit_n}]
By Proposition \ref{without mod p}, $(r,s)\in D(p^n)$ if for all integers $a\geq b\geq 0$ the
following congruence holds
\begin{align}\label{eqn I cri_n}
    \frac{\alpha(a, r)\prod_{j = 0}^{a-1}\alpha(j,p-1)}{\alpha(b, s)\alpha(a-b,r-s)\prod_{j = 0}^{b-1}\alpha(j,p-1)\; \prod_{j = 0}^{a-b-1}\alpha(j,p-1)}\equiv 1 \bmod p^n.
\end{align}
We note that $\alpha(c,m)$ is a unit in $\mathbb{Z}/p^n\mathbb{Z}$ for all $c\in\z$ and
$m\in\p$, as $\alpha(c,m)\equiv 1\bmod p$. If condition $(2)$ holds, then evaluating it at
$x=j$ and using $\alpha(0,r)=1$, we get
$$\alpha(j,p-1)\equiv \frac{\alpha(1,r)\alpha(j,r)}{\alpha(j+1,r)}\bmod p^n\implies
\prod_{j = 0}^{c-1} \alpha(j,p-1)\equiv \frac{\alpha(1,r)^c}{\alpha(c,r)}\bmod p^n\quad\text{for all } c\geq0.$$
If condition $(1)$ holds then $\alpha(c,r)\equiv \alpha(c, s)\equiv \alpha(c,r-s)\bmod p^n$ for
all $c\in \z$. Hence, if both $(1)$ and $(2)$ hold, then the left-hand side of
\eqref{eqn I cri_n} reduces to
$$\frac{\alpha(1,r)^a}{\alpha(1,r)^b \alpha(1,r)^{a-b}}=  1,$$
which shows that $(r,s)\in D(p^n)$. Next, we will prove the converse. If $(r,s)\in D(p^n)$,
then taking $b = 0$ Proposition \ref{without mod p} implies
$\alpha(a, r)\equiv \alpha(a, r-s)\bmod p^n$ for all integers $a\geq0$, and taking $a= b$
gives $\alpha(a, r)\equiv \alpha(a, s)\bmod p^n$. Therefore, we get
$\alpha(a, r)\equiv \alpha(a, s)\equiv\alpha(a,r-s)\bmod p^n$. Now,
Lemma \ref{zero_phi_one} with $\Phi_1(x) = \alpha(x, r)- \alpha(x,s)\bmod p^n$ and
$\Phi_2(x) = \alpha(x, r)- \alpha(x,r-s)\bmod p^n$ implies $(1)$ as
$\deg(\Phi_i)\leq n-1< p-1$. Next, taking $a = b +1$, as $(r,s)\in D(p^n)$,
Proposition \ref{without mod p} gives that
\begin{align*}
    (b+1){r\choose s}\equiv (b+1)\alpha(b, p-1){r\choose s}\frac{\alpha(b+1,r)}{\alpha(b,s)\alpha(1,r-s)}\bmod p^n.
\end{align*}
When $p\nmid (b+1)$, dividing by $(b+1){r\choose s}$ and using $(1)$, we get
$$\alpha(b, p-1)\alpha(b+1,r) \equiv \alpha(1,r)\alpha(b,r)\bmod p^n.$$
Next, let $\Phi(x) = \alpha(x, p-1)\alpha(x+1,r)- \alpha(1,r)\alpha(x,r)\bmod p^n$. Since the
coefficients of $x^k$ in $\alpha(x+1,r),\ \alpha(x,r)$, and $\alpha(x,p-1)$ are divisible by
$p^k$, the same holds for the product $\alpha(x, p-1)\alpha(x+1,r)$. Hence
$\deg(\Phi) \leq n-1< p-1$. Moreover, $\Phi$ vanishes at all integers $b\geq0$ with
$b\not\equiv-1\bmod p$, that is, on $p-1\geq n$ residue classes. Again by
Lemma \ref{zero_phi_one}, we get $\Phi=0$. This completes the proof.
\end{proof}

\begin{proof}[Proof of Corollary \ref{crit_power}]
Comparing the coefficients of $x^k$ in $\alpha(x,r)\equiv \alpha(x, s)\equiv \alpha(x,r-s)\bmod p^n$ and dividing by $p^k$, we see that condition $(1)$ of Theorem \ref{alpha_crit_n} is equivalent to $\h_r(k)\equiv \h_s(k)\equiv \h_{r-s}(k)\bmod p^{n-k}$ for all $1\leq k\leq n-1$, and hence to $(i)$ by Lemma \ref{depth_power}. By Lemma \ref{shift_power}, condition $(2)$ of Theorem \ref{alpha_crit_n} is equivalent to $(ii)$. Hence, conditions $(1)$ and $(2)$ of Theorem \ref{alpha_crit_n} are equivalent to $(i)$ and $(ii)$, and the result follows from Theorem \ref{alpha_crit_n}.
\end{proof}

\section{$D(p^n)$ invariant under symmetries}\label{symmetry_sec}
Pascal's triangle is invariant under one reflection symmetry, namely $\binom rs=\binom r{r-s}$. Further, fixing the number of rows of Pascal's triangle to be $p$, from $0$ to $p-1$, then modulo $p$, up to sign, it is invariant under all the symmetries of the dihedral group $D_3$. This comes naturally from the trinomial coefficients $T(x,y,z)=\frac{(p-1)!}{x!\,y!\,z!}$, which are invariant under all permutations of $x,y$ and $z$. So, fixing one element $(r,s)$ of Pascal's triangle with $p$ rows, one can think of fixing its distances from the three edges,
\[
  x=s,\qquad y=r-s,\qquad z=p-1-r,\qquad x+y+z=p-1,
\]
and the six symmetries of $D_3$ act simply by permuting $x$, $y$ and $z$.
 
\begin{center}\vspace{-4pt}
\begin{tikzpicture}[scale=0.5]
 \draw[gray!60,thick] (0,0) -- (-2.8,-5.2) -- (2.8,-5.2) -- cycle;
 \fill[black] (-0.45,-3) circle (0.075);
 \node[font=\scriptsize,above right] at (-0.8,-2.8) {$(r,s)$};
 \draw[densely dotted,thick,red!70!black] (-0.45,-3) -- (-1.62,-3);
 \draw[densely dotted,thick,red!70!black] (-0.45,-3) -- (1.62,-3);
 \draw[densely dotted,thick,red!70!black] (-0.45,-3) -- (-0.45,-5.2);
 \node[font=\scriptsize,red!70!black,below] at (-1.05,-2.95) {$x$};
 \node[font=\scriptsize,red!70!black,below] at (0.6,-2.95) {$y$};
 \node[font=\scriptsize,red!70!black,left] at (-0.5,-4.2) {$z$};
 \node[font=\scriptsize] at (0,-6.0) {$x+y+z=p-1$};
\end{tikzpicture}\vspace{-6pt}
\end{center}
 
\noindent
Observe that
\begin{align}\label{trinomial eqn}
  \frac{(p-1)!}{x!\,y!\,z!}={p-1\choose z}{x+y\choose x}={p-1\choose x}{y+z\choose y}
  ={p-1\choose y}{z+x\choose z}.
\end{align}
We note that ${p-1\choose k}\equiv(-1)^k \bmod p$. Further, using this in the above equation,
we get
\[
  {r\choose s}\equiv(-1)^{r-s}{{p-1-s}\choose r-s}\equiv(-1)^{s}{{p-1-r+s}\choose s}
  \bmod p .
\]
The reflection is the transposition $x\leftrightarrow y$, which fixes $z$; the common factor
${p-1\choose z}$ then cancels and the identity $\binom rs=\binom r{r-s}$ holds over
$\mathbb Z$. Every other element of $D_3$ moves $z$, so the factors ${p-1\choose k}$ differ
and can only be compared modulo $p$. We illustrate this with $p=7$.
 
\begin{center}
\begin{tikzpicture}[scale=0.54]
\begin{scope}
 \node[align=center,font=\small] at (0,1.35) {$\binom rs \bmod 7$\\[-3pt] one reflection only};
 \draw[gray!35] (0,0.3) -- (-3.15,-6.02) -- (3.15,-6.02) -- cycle;
 \draw[blue!65!black,dashed] (0,0.55) -- (0,-6.55);
 \foreach \r/\vals in {0/{1}, 1/{1,1}, 2/{1,2,1}, 3/{1,3,3,1}, 4/{1,4,6,4,1},
                       5/{1,5,3,3,5,1}, 6/{1,6,1,6,1,6,1}}{
   \foreach \v [count=\s from 0] in \vals {
     \node[font=\small] at ({\s-\r/2},{-\r}) {\v};}}
 \draw[<->,blue!65!black,thin] (-0.75,-4.3) to[bend right=28] (0.75,-4.3);
 \draw[<->,blue!65!black,thin] (-1.28,-5.3) to[bend right=24] (1.28,-5.3);
 \node[font=\scriptsize,blue!65!black] at (0,-7.0) {$\binom rs=\binom r{r-s}$, exactly};
\end{scope}
\begin{scope}[xshift=11cm]
 \node[align=center,font=\small] at (0,1.35) {$(-1)^z\binom rs\equiv T(x,y,z) \bmod 7$\\[-3pt] all six symmetries};
 \draw[gray!35] (0,0.3) -- (-3.15,-6.02) -- (3.15,-6.02) -- cycle;
 \draw[gray!35,dashed] (0,0.3) -- (0,-6.02);
 \draw[gray!35,dashed] (-3.15,-6.02) -- (1.58,-2.86);
 \draw[gray!35,dashed] (3.15,-6.02) -- (-1.58,-2.86);
 \foreach \r/\vals in {0/{1}, 1/{6,6}, 2/{1,2,1}, 3/{6,4,4,6}, 4/{1,4,6,4,1},
                       5/{6,2,4,4,2,6}, 6/{1,6,1,6,1,6,1}}{
   \foreach \v [count=\s from 0] in \vals {
     \node[font=\small] at ({\s-\r/2},{-\r}) {\v};}}
 \foreach \pt in {(-0.5,-3),(0.5,-3),(-1.0,-4),(1.0,-4),(-0.5,-5),(0.5,-5)}
   \draw[orange!85!black,thick] \pt circle (0.30);
 \draw[->,thick,blue!65!black] (2.9,-1.15) arc (32:-88:2.55);
 \node[font=\scriptsize,blue!65!black] at (0,-7.0) {circled: all equal to $4$};
\end{scope}
\end{tikzpicture}\vspace{-8pt}
\end{center}
 
\noindent
In the triangle on the right, the six circled elements form one orbit of $D_3$ their coordinates $(x,y,z)$ are the six permutations of $(1,2,3)$ and have same value $T(1,2,3)=60\equiv4 \bmod 7.$ 
 
Modulo $p$ the congruence above holds for \emph{every} element $(-1)^{p-1-r}{r\choose s}\bmod p$, so by itself it distinguishes nothing. Modulo higher powers of $p$ it fails in general: for $p=7$ the six circled entries become $46,46,4,4,39,39$ modulo $49$, and only four of the $28$ cells retain the symmetry. It is therefore meaningful to ask for which cells it persists, which motivates the following definition.
 
\begin{definition}\label{dfn symmetric}
We say an ordered pair $(r,s)\in\{0,1,\dots,p-1\}^2$ with $r\geq s$ satisfies the rotation symmetry modulo $p^n$, for a positive integer $n$, if
\[
  {r\choose s}\equiv(-1)^{r-s}{{p-1-s}\choose r-s}\equiv(-1)^{s}{{p-1-r+s}\choose s}
  \bmod p^n .
\]
\end{definition}
\begin{rk}
    Let $(r,s)\in\{0,1,...,p-1\}^2$ with $r\geq s$. In \cite[Theorem 8]{rowland} proves that $(r,s)\in D(p^2)$ if and only if $(r,s)$ satisfy the rotation symmetry modulo $p^2$.
\end{rk}

\begin{prop}\label{symmetry_D_p^n}
   Let $1\leq n <p$. If $(r, s)\in D(p^n)$, then $(r,s)$ satisfies the rotational symmetry modulo $p^n$.
\end{prop}
\begin{proof}
For any $k\in\{0,1,\dots,p-1\}$, we have
$${p-1\choose k} = \prod_{i = 1}^{k}\left(\frac{p}{i}-1\right) = (-1)^k\alpha(-1, k).$$
Using the above equation together with \eqref{trinomial eqn} and
${p-1\choose p-1-r}={p-1\choose r}$, we have
\begin{align*}
    \frac{(-1)^{r-s}{p-1-s\choose r-s}}{{r\choose s}}=& \frac{(-1)^{r-s}{p-1\choose r}}{{p-1\choose s}}
     =\frac{\alpha(-1,r)}{\alpha(-1,s)}\\
    \text{and}\quad \frac{(-1)^{r-s}{p-1-s\choose r-s}}{(-1)^s{p-1-r+s\choose p-1-r}}=& \frac{(-1)^{r-s}{p-1\choose r-s}}{(-1)^s{p-1\choose s}}
    =\frac{\alpha(-1,r-s)}{\alpha(-1,s)},
\end{align*}
where the denominators are units, as $\alpha(-1,m)\equiv 1\bmod p$. Since $(r,s)\in D(p^n)$,
condition $(1)$ of Theorem \ref{alpha_crit_n} holds; evaluating it at $x=-1$, we get
$$\alpha(-1,r)\equiv\alpha(-1,s)\equiv\alpha(-1,r-s)\bmod p^n.$$
Hence both ratios above are congruent to $1$ modulo $p^n$. Using
${p-1-r+s\choose p-1-r}={p-1-r+s\choose s}$, we obtain the desired result
\[
  {r\choose s}\equiv(-1)^{r-s}{{p-1-s}\choose r-s}\equiv(-1)^{s}{{p-1-r+s}\choose s}
  \bmod p^n .
\]
\end{proof}
\begin{rk}\label{rk centroid}
     The converse of Proposition \ref{symmetry_D_p^n} does not hold. For $p = 7$, $(4,2)$ the centroid of Pascal's triangle with $7$ rows: it satisfies the rotation symmetry modulo $p^n$ for all $n\geq 1$, yet $(4,2)\not\in D(7^n)$ for all $n\geq 3$.
 \end{rk}
 

Next we prove Theorem \ref{full orbit} which says that if a point $(r,s)$ lies in $D(p^n)$ then its full orbit under the action of the dihedral group $D_3$ lies in $D(p^n)$.

\begin{proof}[Proof of Theorem \ref{full orbit}]
We note that $(1)$ in Theorem \ref{alpha_crit_n} is symmetric in $s$ and $r-s$, and $(2)$
depends only on $r$, so it is the same for the points $(r,s)$ and $(r,r-s)$. Hence, if
$(r,s)\in D(p^n)$, then Theorem \ref{alpha_crit_n} implies that $(r,r-s)\in D(p^n)$. Next, we prove that $(p-1-s,r-s)\in D(p^n)$. Since $(r,s)\in D(p^n)$, by Theorem \ref{alpha_crit_n} we have
\begin{align}\label{eqn r,s full orbit}
 \nonumber &\alpha(x,r)\equiv\alpha(x,s)\equiv\alpha(x, r-s)\bmod p^n\\ 
 &\alpha(x,p-1)\alpha(x+1, r)\equiv\alpha(1,r)\alpha(x,r)\bmod p^n.
\end{align}
Let $m\in\{r,s\}$. By Lemma \ref{alpha_factor}, we
have
$$\alpha(x,p-1-m)\,\alpha(-1-x,m)=\alpha(x,p-1)\,\alpha(-1,m).$$
By \eqref{eqn r,s full orbit}, we have $\alpha(x,m)\equiv\alpha(x,r)\bmod p^n$. Therefore, we get
\begin{align}\label{eqn 1 full orbit}
   \alpha(x,p-1-m)\,\alpha(-1-x,r)\equiv\alpha(x,p-1)\,\alpha(-1,r)\bmod p^n .
\end{align}
We note that \eqref{eqn r,s full orbit} gives the hypothesis of Lemma \ref{alpha_odd}. Hence, by
Lemma \ref{alpha_odd} with $x$ replaced by $x+1$, we get
$\alpha(-1-x,r)\,\alpha(x+1,r)\equiv 1\bmod p^n$. Further, multiplying
\eqref{eqn 1 full orbit} by $\alpha(x+1,r)$, and then using \eqref{eqn r,s full orbit} and Lemma \ref{alpha_odd}
at $x=1$, we get
$$\alpha(x,p-1-m)\equiv\alpha(x,p-1)\,\alpha(x+1,r)\,\alpha(-1,r)\equiv\alpha(1,r)\,\alpha(-1,r)\,\alpha(x,r)\equiv\alpha(x,r)\bmod p^n .$$
Therefore, we have
$$\alpha(x,p-1-s)\equiv\alpha(x,p-1-r)\equiv\alpha(x,r)\equiv\alpha(x,r-s)\bmod p^n .$$
The last congruence follows from \eqref{eqn r,s full orbit}. The above congruence implies $(1)$ in Theorem \ref{alpha_crit_n} for $(p-1-s,r-s)$ holds and also shows that $(2)$ in Theorem \ref{alpha_crit_n} is same for the points $(r,s)$ and $(p-1-s,r-s)$ as $\alpha(x,p-1-s)\equiv \alpha(x,r)\bmod p^n$ (other congruence obtain by replacing $x$ by $x+1$ and $x$). Hence, we get the desired result $(p-1-s,r-s)\in D(p^n)$.
\end{proof}

\section{Special points in $D(p^n)$} 
\subsection{Vertices}\label{vitices sec}
In this subsection, we study a result of Jacobsthal-Ljunggren type congruence in \eqref{eqn Jl type intro} modulo $p^n$ for any positive integer $n<p$ (with our main assumption $p\geq 5$). Such a result gives an equivalent condition for the vertex $(0,0)$ to lie in $D(p^n)$. Now prove Theorem \ref{main_base_gen} which immediately follows from Corollary \ref{crit_power} for $(r,s) = (0,0)$, as $\h_0^{(l)} = 0$ for all $l$.

\begin{proof}[Proof of Corollary \ref{main_base_gen}]
By Corollary \ref{cor_orbit}, the set of vertices $\{(0,0),(p-1,0),(p-1,p-1)\}$ is contained in $D(p^n)$ if and only if $(0,0)\in D(p^n)$. Since $\h^{(l)}_0=0$ for all $l$,  for $(r,s)=(0,0)$ condition $(i)$ of Corollary \ref{crit_power} holds trivially and condition $(ii)$ reduces to $\h_{p-1}^{(l)}\equiv 0\bmod p^{n-l}$ for $1\leq l\leq n-2$. Hence, by Corollary \ref{crit_power}, $(0,0)\in D(p^n)$ if and only if these congruences hold.
\end{proof}

A prime $p$ is called a Wolstenholme prime if ${2p-1\choose p-1}\equiv 1\bmod p^4$. By Glaisher \cite{glaisher1900II} with $n =3$, we have ${2p-1\choose p-1}\equiv 1+2p\,\h_{p-1}^{(1)}\bmod p^4$ for $p\geq5$. Hence $p$ is a Wolstenholme prime if and only if $\h_{p-1}^{(1)}\equiv 0\bmod p^3$ if and only if $B_{p-3}\equiv 0\bmod p$. The last equivalence holds again by Glaisher congruence \cite{glaisher1900}, $\h^{(1)}_{p-1}\equiv-\frac{p^2}{3}B_{p-3}\bmod p^3$. Only two Wolstenholme primes $16843$ and $2124679$ are known and by \cite{BHMT} these are only two Wolstenholme primes less than $10^{11}$. 
\begin{cor}\label{vertex_small_n}
\begin{enumerate}
  \item[(a)] For $p\geq 5$, $(0,0)\in D(p^3)$.
  \item[(b)] For $p\geq 5$, $(0,0)\in D(p^4)$ if and only if $\h_{p-1}^{(1)}\equiv 0\bmod p^3$ if and only if $p$ is a Wolstenholme prime.
  \item[(c)] For $p\geq7$, $(0,0)\in D(p^5)$ if and only if
  $\h^{(1)}_{p-1}\equiv0\bmod p^4$, if and only if ${2p-1\choose p-1}\equiv1\bmod p^5$.
\end{enumerate}
\end{cor}
\begin{proof}
We first note that the lower bound on primes comes directly from the hypothesis $n< p$ in Corollary \ref{main_base_gen}. By Corollary \ref{main_base_gen}, $(0,0)\in D(p^3)$ if and only if $\h_{p-1}^{(1)}\equiv 0\bmod p^2$, which holds for all primes $p\geq 5$ by Wolstenholme theorem (see Theorem \ref{Wolstenholme}) which proves $(a)$. Next, by \eqref{eqn pair i_p-i} in Proof of Lemma \ref{valuation} with $k =1$ gives 
\begin{equation}\label{pair_full}
\h^{(1)}_{p-1}\equiv-\frac12\sum_{t = 1}^{n-1}p^t\,\h^{(t+1)}_{p-1}\bmod p^n.
\end{equation}
Taking $n = 3$, we obtain $\h^{(1)}_{p-1}\equiv-\frac p2\,\h^{(2)}_{p-1}\bmod p^3$ as $p^2\h^{(3)}_{p-1}\equiv0\bmod p^3$ follows from Lemma \ref{valuation}. Hence the two conditions  of Corollary \ref{main_base_gen} reduce to $\h_{p-1}^{(1)}\equiv 0\bmod p^3$ which proves $(b)$.

Next, we prove $(c)$. By Lemma \ref{valuation} $\h^{(3)}_{p-1}\equiv0\bmod p^2$ and $\h^{(4)}_{p-1}\equiv0\bmod p$, which implies $p^2\h_{p-1}^{(3)} + p^3\h_{p-1}^{(4)}\equiv 0\bmod p^4$. Further \eqref{pair_full} gives us $\h_{p-1}^{(1)}\equiv -\frac 12 p\h_{p-1}^{(2)}\bmod p^4$. Hence, by Corollary \ref{main_base_gen} the first equivalence holds as the condition $\h_{p-1}^{(2)}\equiv 0\bmod p^3$ reduce to $\h_{p-1}^{(1)}\equiv 0\bmod p^4$ and the last condition $\h_{p-1}^{(3)}\equiv 0\bmod p^2$ always holds. By \cite[Theorem~3.2]{Zhao} we get
\begin{align*}
    \frac{{2p\choose p}}{{2\choose 1}}&\equiv 1+2p^3\omega_p\bmod p^5,\quad\text{where } \omega_p\equiv \h_{p-1}^{(1)}/p^2\bmod p^2\\
    \implies {2p-1\choose p-1}&\equiv 1+2p\h_{p-1}^{(1)}\bmod p^5.
\end{align*}
Hence ${2p-1\choose p-1}\equiv1\bmod p^5$ if and only if $\h_{p-1}^{(1)}\equiv 0\bmod p^4$,
which gives the last equivalence in $(c)$.
\end{proof}

\subsection{The centroid}\label{centroid sec}
In this subsection, we study for which primes $p$ the centroid belongs to $D(p^n)$. Remark \ref{rk centroid} shows that in general it does not. In the rest of the subsection, we assume that $p\equiv 1\bmod 3$. We have the following result.


\begin{theorem}\label{centroid_n}
Let $p\equiv1\bmod3$, $m=\frac{p-1}{3}$, and assume $n<p$. Then $(2m,m)\in D(p^n)$ if and only if
\begin{enumerate}
  \item[(a)] $\h^{(l)}_{2m}\equiv\h^{(l)}_m\bmod p^{n-l}$ for all $1\leq l\leq n-1$; and
  \item[(b)] $\h^{(l)}_m\equiv0\bmod p^{n-l}$ for all even $l$ with $2\leq l\leq n-1$.
\end{enumerate}
\end{theorem}
\begin{proof}
By Corollary \ref{crit_power} $(2m,m)\in D(p^n)$ if and only if 
\begin{align*}
    \nonumber \h_{2m}^{(l)}&\equiv\h_m^{(l)}\bmod p^{n-l}\quad\text{for all } 1\leq l\leq n-1\\
   \text{and}\quad \h_{p-1}^{(i)}&\equiv\sum_{t =1}^{n-i-1}(-1)^{t+1}{i+t-1\choose t}p^t\h_{2m}^{(i+t)}\bmod p^{n-i}\quad\text{for all } 1\leq i\leq n-2.
\end{align*}
The first congruence is  exactly the same as $(a)$ and further using $(a)$ second reduce to the following
\begin{align}\label{eqn centroid_n I}
    \h_{p-1}^{(i)}\equiv\sum_{t =1}^{n-i-1}(-1)^{t+1}{i+t-1\choose t}p^t\h_{m}^{(i+t)}\bmod p^{n-i}\quad\text{for all } 1\leq i\leq n-2.
\end{align}
Next, we will show that $(b)$ is equivalent to the above congruence. Using expansion of $(p-i)^i$ for any $i$ as used to deduce \eqref{eqn pair i_p-i} for odd $k$ in the proof of Lemma \ref{valuation}, we get 
\begin{align}\label{eqn paring i_p-i_any_i}
    \h_{p-1}^{(i)} = \h_{2m}^{(i)}+\sum_{j = 1}^{m}\frac{1}{(p-j)^i}\equiv\h_{2m}^{(i)}+(-1)^i\sum_{t = 0}^{n-i-1}{i+t-1\choose t}p^t\h_m^{(i+t)}\bmod p^{n-i}.
\end{align}
Further the above congruence together with $(a)$ implies that \eqref{eqn centroid_n I} is equivalent to the following
\begin{align}\label{cen_even}
\nonumber& \sum_{t=0}^{n-i-1}\bigl((-1)^i+(-1)^t\bigr){i+t-1\choose t}\,p^t\,\h^{(i+t)}_m\equiv0\bmod p^{n-i}\\
\iff & 2(-1)^i\sum_{t=0,\; i+t \text{ even}}^{n-i-1}{i+t-1\choose t}\,p^t\,\h^{(i+t)}_m\equiv0\bmod p^{n-i}\quad\text{for all }1\leq i\leq n-2.
\end{align}
If (b) holds, then $p^t\h^{(i+t)}_m\equiv0\bmod p^{n-i}$ for all $0\leq t\leq n-i-1$ with $i+t$ even. Hence \eqref{cen_even} holds.

Conversely, suppose \eqref{cen_even} holds for all $1\leq i\leq n-2$. We show $\h^{(l)}_m\equiv0\bmod p^{n-l}$ for every even $2\leq l\leq n-1$ by downward induction on $l$. Take $i=l-1$ in \eqref{cen_even}, which is odd with $1\leq i\leq n-2$. Then only odd $t$ contribute, and dividing by $-2p$ we get
$$(l-1)\,\h^{(l)}_m+\sum_{\substack{t=3\\ t \text{ odd}}}^{n-l}{l+t-2\choose t}\,p^{t-1}\,\h^{(l+t-1)}_m\equiv0\bmod p^{n-l}.$$
Let $l_0\leq n-1$ be the largest even integer. This implies that $n-1< l_0+1$ giving summation empty when $l = l_0$. Hence, taking $l = l_0$ gives $\h_m^{(l_0)}\equiv 0\mod p^{n-l_0}$ which proves the claim for $l=l_0$. Next let $l< l_0$. Each $l+t-1$ is even and greater than $l$ as $t\geq 3$, so by induction $p^{t-1}\h^{(l+t-1)}_m\equiv0\bmod p^{(t-1)+(n-l-t+1)}=p^{n-l}$ for all odd $t\geq 3$, implies that the summation vanishes modulo $p^{n-l}$. Hence, we get $\h^{(l)}_m\equiv0\bmod p^{n-l}$ as $l-1\leq n-2< p$ is a unit. This proves (b).
\end{proof}

\begin{rk}\label{rk_h_bernoulli}
For general $n\leq p-1$, Theorem \ref{centroid_n} can be expressed entirely in terms of Bernoulli polynomials. For $i\leq p-1$, since $(1-i^{p-1})^{n-l}\equiv0\bmod p^{n-l}$, one has $i^{-l}\equiv\sum_{s=1}^{n-l}(-1)^{s-1}{n-l\choose s}\,i^{\,s(p-1)-l}\bmod p^{n-l}$, and summing over $1\leq i\leq c$ with $c\in\{m,2m\}$ and using Theorem \ref{bernoulli poly} gives, with $a_s=s(p-1)-l+1$,
$$\h^{(l)}_c\equiv\sum_{s=1}^{n-l}(-1)^{s-1}{n-l\choose s}\frac{B_{a_s}(c+1)-B_{a_s}}{a_s}\bmod p^{n-l},
  \qquad c\in\{m,2m\}.$$\pagebreak\\
\end{rk}

Next we prove Corollary \ref{centroid_small_n}. 
\begin{proof}[Proof of Corollary \ref{centroid_small_n} $(1)$]
By Theorem \ref{Wolstenholme} we have $\h^{(1)}_{p-1}\equiv0\bmod p^2$ and $\h^{(2)}_{p-1}\equiv0\bmod p$. These together with \eqref{eqn paring i_p-i_any_i} for $i =1,2$ and $n =3$ give
$$\h^{(1)}_{2m}\equiv\h^{(1)}_m+p\,\h^{(2)}_m\bmod p^2\quad\text{and}\quad
\h^{(2)}_{2m}\equiv-\h^{(2)}_m\bmod p .$$
Hence the two conditions $\h^{(1)}_{2m}\equiv\h^{(1)}_m\bmod p^2$, $\h^{(2)}_{2m}\equiv\h^{(2)}_m\bmod p$ are equivalent to $\h^{(2)}_m\equiv0\bmod p$. By Theorem \ref{centroid_n} with $n=3$, $(2m,m)\in D(p^3)$ if and only if  $\h^{(2)}_m\equiv0\bmod p$. Finally, by Remark \ref{rk_h_bernoulli} with $n=3$ and $l=2$, $\h^{(2)}_m\equiv\frac12B_{p-2}(1/3)\bmod p$ which proves the desired result.
\end{proof}

\begin{proof}[Proof of Corollary \ref{centroid_small_n} $(2)$]
By Theorem \ref{centroid_n}, $(2m,m)\in D(p^4)$ if and only if 
\begin{align}\label{eqn centroid_n4_1}
   \h_{m}^{(2)}\equiv 0\bmod p^2\quad \text{and } \h_{2m}^{(l)}\equiv\h_{m}^{(l)}\bmod p^{4-l}\quad\text{for all } 1\leq l\leq 3. 
\end{align}
We note that the first congruence is part of the statement, and the second congruence for $l=3$ always holds. This can be seen because Lemma \ref{valuation} gives $\h^{(3)}_{p-1}\equiv0\bmod p$ and this together with \eqref{eqn paring i_p-i_any_i} for $n=4$ and $i=3$, gives $\h^{(3)}_{2m}\equiv\h^{(3)}_m\bmod p$. Next, assume $\h^{(2)}_m\equiv0\bmod p^2$, which is one of the conditions in the statement. We will show that the two conditions $\h_{2m}^{(l)}\equiv\h_m^{(l)}\bmod p^{4-l}$ for $l=1,2$ are equivalent to one condition that $p$ is Wolstenholme prime.

Again applying \eqref{eqn paring i_p-i_any_i} with $n=4$ and $i=1,2$ we get
\begin{align*}
\h^{(2)}_{2m}&\equiv\h^{(2)}_{p-1}-\h^{(2)}_m-2p\,\h^{(3)}_m\bmod p^2,\\
\h^{(1)}_{2m}&\equiv\h^{(1)}_{p-1}+\h^{(1)}_m+p\,\h^{(2)}_m+p^2\,\h^{(3)}_m\bmod p^3 .
\end{align*}
Assuming $\h^{(2)}_m\equiv0\bmod p^2$, the two conditions $\h_{2m}^{(l)}\equiv\h_m^{(l)}\bmod p^{4-l}$ for $l=1,2$ equivalent to the following two conditions
$$\h^{(2)}_{p-1}\equiv2p\,\h^{(3)}_m\bmod p^2\quad\text{and}\quad
\h^{(1)}_{p-1}\equiv-p^2\,\h^{(3)}_m\bmod p^3.$$
By Remark \ref{rk_h_bernoulli} with $n=4$, $l=3$ and $c=m$, $c+1\equiv \frac23\bmod p$, we get
\begin{align}\label{eqn centroid_n4_2}
    \h^{(3)}_m\equiv \frac{1}{p-3}\bigl(B_{p-3}(2/3)- B_{p-3}\bigr)\bmod p.
\end{align}
By Raabe's multiplication theorem \cite[\S4, p.~356]{raabe}, we have $B_k(dx)=d^{k-1}\sum_{i=0}^{d-1}B_k\bigl(x+\frac{i}{d}\bigr)$ for any integer $d\geq1$. Applying this with $x=0$ and $d=3$, and using $B_k(\tfrac23)=B_k(\tfrac13)$ for even $k$, we get $B_k(\tfrac23)=\frac12(3^{1-k}-1)B_k$. For $k=p-3$, since $3^{4-p}\equiv27\bmod p$, this gives $B_{p-3}(2/3)\equiv 13 B_{p-3}\bmod p$. Therefore, \eqref{eqn centroid_n4_2} gives $\h^{(3)}_m\equiv-4B_{p-3}\bmod p$. Hence the two conditions become
$$\h^{(2)}_{p-1}\equiv -8pB_{p-3}\bmod p^2\quad\text{and}\quad
\h^{(1)}_{p-1}\equiv 4p^2B_{p-3}\bmod p^3.$$
By Glaisher's congruences \cite{glaisher1900}, $\h^{(2)}_{p-1}\equiv\frac{2p}{3}B_{p-3}\bmod p^2$ and $\h^{(1)}_{p-1}\equiv-\frac{p^2}{3}B_{p-3}\bmod p^3$. Substituting, both conditions become $\frac{13}{3}B_{p-3}\equiv0\bmod p$. For $p=13$, a direct computation gives $\h^{(2)}_{m}=\frac{205}{144}\not\equiv0\bmod 13$, so both sides of the claim fail. For $p\neq13$, $\frac{13}{3}B_{p-3}\equiv0\bmod p$ if and only if $B_{p-3}\equiv0\bmod p$, i.e., $p$ is a Wolstenholme prime. This proves the desired claim.
\end{proof}


\begin{rk}[computational]
For every prime $5\leq p\leq 250000$,
\[
    |D(p^3)| = \begin{cases}
        3 & \text{if } p \neq 205129,\\
        4 & \text{if } p = 205129,
    \end{cases}
\]
the three elements being the vertices $(0,0)$, $(p-1,0)$, $(p-1,p-1)$ in all cases, and the fourth element for $p = 205129$ being the centroid $(136752, 68376)$.
\end{rk}

\subsection{The edge midpoints}\label{edge midpoint sec}
In this subsection, we will study for which primes the edge midpoints of Pascal's triangle with $p$ rows belong to $D(p^n)$. Let $e = (p-1)/2$. Then the set of edge midpoints is given by
$$\{(e,0),\ (e, e),\ (p-1, e)\},$$
which is a single orbit under the action $D_3$ by Corollary \ref{cor_orbit}. We have the following result for edge mid-points. 

\begin{theorem}\label{edge-midpoints}
Let $e=\frac{p-1}{2}$ and $n<p$. Then the edge midpoints lie in $D(p^n)$ if and only if
$$\h^{(l)}_e\equiv0\bmod p^{n-l}\qquad\text{for all } 1\leq l\leq n-1 .$$
Moreover, in this case $(0,0)\in D(p^n)$.
\end{theorem}
\begin{proof}
By Corollary \ref{cor_orbit}, the edge midpoints lie in $D(p^n)$ if and only if $(e,0)\in D(p^n)$. Further, Theorem \ref{alpha_crit_n} implies that $(e,0)\in D(p^n)$ if and only if 
$$\alpha(x,e)\equiv1\bmod p^n\quad\text{and}\quad \alpha(x,p-1)\,\alpha(x+1,e)\equiv\alpha(1,e)\,\alpha(x,e)\bmod p^n .$$
Suppose $\alpha(x,e)\equiv 1\bmod p^n$ holds. Substituting $x+1$ and $1$ for $x$, we also have $\alpha(x+1,e)\equiv\alpha(1,e)\equiv1\bmod p^n$. Hence, the second congruence reduces to $\alpha(x,p-1)\equiv1\bmod p^n$. Moreover, this congruence follows from $\alpha(x,e)\equiv1\bmod p^n$. Indeed, Lemma \ref{alpha_factor}$(2)$ with $m=e$, so that $p-1-m=e$, gives $\alpha(x,e)\,\alpha(-1-x,e)=\alpha(x,p-1)\,\alpha(-1,e)$, where $\alpha(x,e)\equiv \alpha(-1-x,e)\equiv\alpha(-1,e) \equiv1\bmod p^n$. Hence $(e,0)\in D(p^n)$ if and only if $\alpha(x,e)\equiv1\bmod p^n$. Comparing the coefficients of $x^k$ and dividing by $p^k$, this says $\h_e(k)\equiv0\bmod p^{n-k}$ for $1\leq k\leq n-1$. By Lemma \ref{depth_power} with $j'=0$ and $j=e$, this is equivalent to
$$\h^{(l)}_e\equiv0\bmod p^{n-l}\qquad\text{for all } 1\leq l\leq n-1 .$$
Finally, as seen above, $(e,0)\in D(p^n)$ implies $\alpha(x,p-1)\equiv1\bmod p^n$, which by Theorem \ref{alpha_crit_n} with $(r,s)=(0,0)$ is equivalent to $(0,0)\in D(p^n)$.
\end{proof}

Next, we prove Corollary \ref{edge midpoint_p^3}. 
\begin{proof}[Proof of Corollary \ref{edge midpoint_p^3}]
By Theorem \ref{edge-midpoints} with $n=3$, the edge midpoints lie in $D(p^3)$ if and only if
\begin{align}\label{eqn edge-midpoint_n2}
    \h_e^{(1)}\equiv 0\bmod p^2\quad\text{and}\quad \h_e^{(2)}\equiv 0\bmod p.
\end{align}
By pairing $i$ with $p-i$ and using $(p-i)^{-2}\equiv i^{-2}\bmod p$, we get $\h^{(2)}_{p-1}\equiv2\,\h^{(2)}_e\bmod p$ and $\h^{(2)}_{p-1}\equiv 0\bmod p$ by Wolstenholme theorem, as $p\geq 5$. Hence, second congruence in \eqref{eqn edge-midpoint_n2} always holds. This proves the first equivalence.

Next, by Lehmer's congruence (Theorem \ref{Lehmer}), we have
\begin{align}\label{lehmer eqn}
    \h^{(1)}_e\equiv -2q_p(2)+p\,q_p(2)^2\bmod{p^2}.
\end{align}
If $\h^{(1)}_e\equiv 0 \bmod p^2$, then reducing \eqref{lehmer eqn} modulo $p$ gives
$q_p(2)\equiv 0 \bmod p$. Let $t$ be an integer such that $q_p(2)=pt$. Substituting into
\eqref{lehmer eqn}, we obtain $0\equiv-2pt+p^3t^2\equiv-2pt\bmod p^2$, so
$t\equiv 0\bmod p$. Hence $q_p(2)\equiv 0\bmod{p^2}$, which implies that
$2^{p-1}=1+p\,q_p(2)\equiv 1 \bmod p^3$. Conversely, if $2^{p-1}\equiv1\bmod p^3$, then
$q_p(2)\equiv0\bmod p^2$, and both terms on the right of \eqref{lehmer eqn} vanish modulo
$p^2$, so $\h^{(1)}_e\equiv0\bmod p^2$. This proves the second equivalence.
\end{proof}

\begin{cor}\label{edge_n4_free}
The edge midpoints lie in $D(p^4)$ if and only if $p$ is a Wolstenholme prime and
$2^{p-1}\equiv1\bmod p^4$.
\end{cor}

\begin{proof}
By Theorem \ref{edge-midpoints} with $n=4$, the edge midpoints lie in $D(p^4)$ if and only if
\begin{align}\label{eqn edge_mid_pt_n4}
    \h^{(1)}_e\equiv0\bmod p^3,\; \h^{(2)}_e\equiv0\bmod p^2\quad\text{and}\quad \h^{(3)}_e\equiv0\bmod p.
\end{align}
Arguing as in Remark \ref{rk_h_bernoulli} with $e$ in place of $m$ and using $e+1\equiv\frac12\bmod p$,
we get $\h^{(3)}_e\equiv\frac{B_{p-3}(1/2)-B_{p-3}}{p-3}\bmod p$. By Raabe's multiplication theorem
\cite[\S4, p.~356]{raabe}, we have $B_k(dx)=d^{k-1}\sum_{i=0}^{d-1}B_k\bigl(x+\frac{i}{d}\bigr)$ for any integer $d\geq1$. Applying this with $x=0$ and $d=2$, we get $B_k=2^{k-1}\bigl(B_k+B_k(\tfrac12)\bigr)$, i.e., $B_k(\tfrac12)=(2^{1-k}-1)B_k$. For $k=p-3$, since $B_{p-3}$ is $p$-integral and $2^{4-p}\equiv8\bmod p$, this gives $B_{p-3}(\tfrac12)\equiv7B_{p-3}\bmod p$. Therefore, we get $\h^{(3)}_e\equiv-2B_{p-3}\bmod p$.
So the last condition in \eqref{eqn edge_mid_pt_n4} says that $B_{p-3}\equiv0\bmod p$, i.e., $p$ is a Wolstenholme prime. We assume this for rest for the proof. Then Glaisher's congruence $\h^{(2)}_{p-1}\equiv\frac{2p}{3}B_{p-3}\bmod p^2$ gives $\h^{(2)}_{p-1}\equiv0\bmod p^2$. Pairing $i$ with $p-i$ gives
$\h^{(2)}_{p-1}\equiv2\,\h^{(2)}_e+2p\,\h^{(3)}_e\bmod p^2$, so $\h^{(2)}_e\equiv0\bmod p^2$, and the
second condition in \eqref{eqn edge_mid_pt_n4} holds automatically. Finally, by Sun \cite[Theorem 5.2(c)]{Sun2000},
$$\h^{(1)}_e\equiv-2q+pq^2-\tfrac23p^2q^3-\tfrac{7}{12}p^2B_{p-3}\bmod p^3,\qquad q=q_p(2),$$
and the last term vanishes as $B_{p-3}\equiv 0\bmod p$. Reducing modulo $p$, both $\h^{(1)}_e\equiv0\bmod p^3$
and $2^{p-1}\equiv1\bmod p^4$ imply $q\equiv0\bmod p$. In that case $pq^2\equiv p^2q^3\equiv0\bmod p^3$,
so $\h^{(1)}_e\equiv-2q\bmod p^3$. Hence $\h^{(1)}_e\equiv0\bmod p^3$ if and only if
$q\equiv0\bmod p^3$, i.e., $2^{p-1}=1+pq\equiv1\bmod p^4$.
\end{proof}


{\bf Acknowledgements.} We acknowledge the work in \cite{rowland}, which provides the basis for our generalization. The author thanks Shalini Bhattacharya for suggesting the problem and for helpful discussions. The author also acknowledges that this work was initiated during his postdoctoral tenure at IISER Tirupati and gratefully acknowledges the institute's support during that period.\vspace{.5em}\\

{\bf AI Declaration.} The results for \(n=3\) were obtained without the assistance of AI tools. Claude (Anthropic) was subsequently used to explore ideas for generalizing these results to arbitrary \(n\). All proofs were written by the author, using AI tools to improve the language and presentation of the manuscript.
 

\bibliographystyle{plain}
\bibliography{lucas.bib}

@article {rowland,
    AUTHOR = {Rowland, Eric},
     TITLE = {Lucas' theorem modulo {$p^2$}},
   JOURNAL = {Amer. Math. Monthly},
  FJOURNAL = {American Mathematical Monthly},
    VOLUME = {129},
      YEAR = {2022},
    NUMBER = {9},
     PAGES = {846--855},
}

@article{Wolstenholme,
    author = {Wolstenholme, J.},
    title = {On certain properties of prime numbers},
    journal = {Q. J. Pure Appl. Math.},
    VOLUME = {5},
    year = {1862},
    PAGES = {35--39},
}

@article {sun2010,
    AUTHOR = {Sun, Zhiwei},
     TITLE = {New series for some special values of {$L$}-functions},
   JOURNAL = {Nanjing Daxue Xuebao Shuxue Bannian Kan},
  FJOURNAL = {Nanjing Daxue Xuebao. Shuxue Bannian Kan. Nanjing University.
              Journal. Mathematical Biquarterly},
    VOLUME = {32},
      YEAR = {2015},
    NUMBER = {2},
     PAGES = {189--218},
      ISSN = {0469-5097},
   MRCLASS = {11B65 (11M06)},
  MRNUMBER = {3616300},
}

@article{Romeo,
    author = {Me\v{s}trovi\'c, Romeo},
    title = {Lucas’ theorem: its generalizations, extensions and applications (1878– 2014)},
    journal = {arXiv \url{https://arxiv.org/abs/1409.3820}},
    VOLUME = {},
    year = {2014},
    PAGES = {},
}

@article {lucas,
    AUTHOR = {Lucas, E.},
     TITLE = {Sur les congruences des nombres eul\'eriens et les
              coefficients diff\'erentiels des functions trigonom\'etriques
              suivant un module premier},
   JOURNAL = {Bull. Soc. Math. France},
  FJOURNAL = {Bulletin de la Soci\'et\'e{} Math\'ematique de France},
    VOLUME = {6},
      YEAR = {1878},
     PAGES = {49--54},
      ISSN = {0037-9484},
   MRCLASS = {99-04},
  MRNUMBER = {1503769},
       URL = {http://www.numdam.org/item?id=BSMF_1878__6__49_1},
}

@article {Zhao,
    AUTHOR = {Zhao, Jianqiang},
     TITLE = {Bernoulli numbers, {W}olstenholme's theorem, and {$p^5$}
              variations of {L}ucas' theorem},
   JOURNAL = {J. Number Theory},
  FJOURNAL = {Journal of Number Theory},
    VOLUME = {123},
      YEAR = {2007},
    NUMBER = {1},
     PAGES = {18--26},
      ISSN = {0022-314X,1096-1658},
   MRCLASS = {11A07 (11B68)},
  MRNUMBER = {2295427},
MRREVIEWER = {Herman\ J. J. te Riele},
       DOI = {10.1016/j.jnt.2006.05.005},
       URL = {https://doi.org/10.1016/j.jnt.2006.05.005},
}

@article {McIntosh,
    AUTHOR = {McIntosh, Richard J.},
     TITLE = {On the converse of {W}olstenholme's theorem},
   JOURNAL = {Acta Arith.},
  FJOURNAL = {Acta Arithmetica},
    VOLUME = {71},
      YEAR = {1995},
    NUMBER = {4},
     PAGES = {381--389},
      ISSN = {0065-1036,1730-6264},
   MRCLASS = {11A07 (11B68)},
  MRNUMBER = {1339137},
MRREVIEWER = {Joe\ P.\ Buhler},
       DOI = {10.4064/aa-71-4-381-389},
       URL = {https://doi.org/10.4064/aa-71-4-381-389},
}

@article {BHMT,
    AUTHOR = {Booker, Andrew R. and Hathi, Shehzad and Mossinghoff, Michael
              J. and Trudgian, Timothy S.},
     TITLE = {Wolstenholme and {V}andiver primes},
   JOURNAL = {Ramanujan J.},
  FJOURNAL = {Ramanujan Journal. An International Journal Devoted to the
              Areas of Mathematics Influenced by Ramanujan},
    VOLUME = {58},
      YEAR = {2022},
    NUMBER = {3},
     PAGES = {913--941},
      ISSN = {1382-4090,1572-9303},
   MRCLASS = {11B68 (11A41 11B65 11Y40)},
  MRNUMBER = {4437433},
MRREVIEWER = {Samuel\ S.\ Wagstaff, Jr.},
       DOI = {10.1007/s11139-021-00438-3},
       URL = {https://doi.org/10.1007/s11139-021-00438-3},
}

@article {murty,
    AUTHOR = {Murty, M. Ram and S\'eguin, Fran\c cois},
     TITLE = {Prime divisors of sparse values of cyclotomic polynomials and
              {W}ieferich primes},
   JOURNAL = {J. Number Theory},
  FJOURNAL = {Journal of Number Theory},
    VOLUME = {201},
      YEAR = {2019},
     PAGES = {1--22},
      ISSN = {0022-314X,1096-1658},
   MRCLASS = {11B39 (11A41 11D45 11N69)},
  MRNUMBER = {3958037},
MRREVIEWER = {Joshua\ Zelinsky},
       DOI = {10.1016/j.jnt.2019.02.016},
       URL = {https://doi.org/10.1016/j.jnt.2019.02.016},
}

@article {sun2000,
    AUTHOR = {Sun, Zhi-Hong},
     TITLE = {Congruences concerning {B}ernoulli numbers and {B}ernoulli
              polynomials},
   JOURNAL = {Discrete Appl. Math.},
  FJOURNAL = {Discrete Applied Mathematics. The Journal of Combinatorial
              Algorithms, Informatics and Computational Sciences},
    VOLUME = {105},
      YEAR = {2000},
    NUMBER = {1-3},
     PAGES = {193--223},
      ISSN = {0166-218X,1872-6771},
   MRCLASS = {11B68 (11A07)},
  MRNUMBER = {1780472},
MRREVIEWER = {Arnold\ M.\ Adelberg},
       DOI = {10.1016/S0166-218X(00)00184-0},
       URL = {https://doi.org/10.1016/S0166-218X(00)00184-0},
}

@article{glaisher1900,
  title={ Congruences relating to the sums of products of the first n numbers and to other sums of products},
  author={Glaisher, James Whitbread Lee},
  journal={Quart. J. Math},
  volume={31},
  pages={1-35},
  year={1900}
}

@article{glaisher1900II,
  title={On the residues of the sums of products of the first p- 1 numbers, and their powers, to modulus p2 or p3},
  author={Glaisher, James Whitbread Lee},
  journal={Quart. J. Math},
  volume={31},
  pages={321--353},
  year={1900}
}

@article{raabe,
  author={Raabe, Joseph Ludwig},
  title={Zur{\"u}ckf{\"u}hrung einiger Summen und bestimmten Integrale auf die Jacob-Bernoullische Function.},
  
  year={1851},
  JOURNAL ={De Gruyter},
}

@article{babbage,
    author = {Babbage, Charles} ,
    title = {Demonstration of a theorem relating to prime numbers},
    journal = {The Edinburgh Philosophical Journal},
    volume= {1},
    year = {1819},
    pages={46-49},
}

@incollection {Lj,
    AUTHOR = {Brun, Viggo and Stubban, J. O. and Fjeldstad, J. E. and Tambs
              Lyche, R. and Aubert, K. E. and Ljunggren, W. and Jacobsthal,
              E.},
     TITLE = {On the divisibility of the difference between two binomial
              coefficients},
 BOOKTITLE = {Den 11te {S}kandinaviske {M}atematikerkongress, {T}rondheim,
              1949},
     PAGES = {42--54},
 PUBLISHER = {Johan Grundt Tanums Forlag, Oslo},
      YEAR = {1952},
   MRCLASS = {10.0X},
  MRNUMBER = {53125},
MRREVIEWER = {D.\ H.\ Lehmer},
}

@article {Bailey,
    AUTHOR = {Bailey, D. F.},
     TITLE = {Two {$p^3$} variations of {L}ucas' theorem},
   JOURNAL = {J. Number Theory},
  FJOURNAL = {Journal of Number Theory},
    VOLUME = {35},
      YEAR = {1990},
    NUMBER = {2},
     PAGES = {208--215},
      ISSN = {0022-314X,1096-1658},
   MRCLASS = {11B39},
  MRNUMBER = {1057323},
MRREVIEWER = {A.\ Eswarathasan},
       DOI = {10.1016/0022-314X(90)90113-6},
       URL = {https://doi.org/10.1016/0022-314X(90)90113-6},
}

@incollection {Granville,
    AUTHOR = {Granville, Andrew},
     TITLE = {Arithmetic properties of binomial coefficients. {I}.
              {B}inomial coefficients modulo prime powers},
 BOOKTITLE = {Organic mathematics ({B}urnaby, {BC}, 1995)},
    SERIES = {CMS Conf. Proc.},
    VOLUME = {20},
     PAGES = {253--276},
 PUBLISHER = {Amer. Math. Soc., Providence, RI},
      YEAR = {1997},
      ISBN = {0-8218-0668-8},
   MRCLASS = {11B65 (11A07)},
  MRNUMBER = {1483922},
MRREVIEWER = {Ronald\ J.\ Evans},
}

@article {davis-webb,
    AUTHOR = {Davis, Kenneth S. and Webb, William A.},
     TITLE = {Lucas' theorem for prime powers},
   JOURNAL = {European J. Combin.},
  FJOURNAL = {European Journal of Combinatorics},
    VOLUME = {11},
      YEAR = {1990},
    NUMBER = {3},
     PAGES = {229--233},
      ISSN = {0195-6698,1095-9971},
   MRCLASS = {11B65 (05A10)},
  MRNUMBER = {1059553},
MRREVIEWER = {F.\ T.\ Howard},
       DOI = {10.1016/S0195-6698(13)80122-9},
       URL = {https://doi.org/10.1016/S0195-6698(13)80122-9},
}

@article {lehmer,
    AUTHOR = {Lehmer, Emma},
     TITLE = {On congruences involving {B}ernoulli numbers and the quotients
              of {F}ermat and {W}ilson},
   JOURNAL = {Ann. of Math. (2)},
  FJOURNAL = {Annals of Mathematics. Second Series},
    VOLUME = {39},
      YEAR = {1938},
    NUMBER = {2},
     PAGES = {350--360},
      ISSN = {0003-486X,1939-8980},
   MRCLASS = {99-04},
  MRNUMBER = {1503412},
       DOI = {10.2307/1968791},
       URL = {https://doi.org/10.2307/1968791},
}

@book {Kenneth-Rosen,
    AUTHOR = {Ireland, Kenneth and Rosen, Michael},
     TITLE = {A classical introduction to modern number theory},
    SERIES = {Graduate Texts in Mathematics},
    VOLUME = {84},
   EDITION = {Second},
 PUBLISHER = {Springer-Verlag, New York},
      YEAR = {1990},
     PAGES = {xiv+389},
      ISBN = {0-387-97329-X},
   MRCLASS = {11-01 (11-02)},
  MRNUMBER = {1070716},
MRREVIEWER = {Glenn\ Stevens},
       DOI = {10.1007/978-1-4757-2103-4},
       URL = {https://doi.org/10.1007/978-1-4757-2103-4},
}

@book {Neukirch,
    AUTHOR = {Neukirch, J\"urgen},
     TITLE = {Algebraic number theory},
    SERIES = {Grundlehren der mathematischen Wissenschaften [Fundamental
              Principles of Mathematical Sciences]},
    VOLUME = {322},
      NOTE = {Translated from the 1992 German original and with a note by
              Norbert Schappacher,
              With a foreword by G. Harder},
 PUBLISHER = {Springer-Verlag, Berlin},
      YEAR = {1999},
     PAGES = {xviii+571},
      ISBN = {3-540-65399-6},
   MRCLASS = {11Rxx (11-02 11S15 11S31 14C40)},
  MRNUMBER = {1697859},
MRREVIEWER = {Cornelius\ Greither},
       DOI = {10.1007/978-3-662-03983-0},
       URL = {https://doi.org/10.1007/978-3-662-03983-0},
}

@article {Davis-Webb1993,
    AUTHOR = {Davis, Kenneth and Webb, William},
     TITLE = {A binomial coefficient congruence modulo prime powers},
   JOURNAL = {J. Number Theory},
  FJOURNAL = {Journal of Number Theory},
    VOLUME = {43},
      YEAR = {1993},
    NUMBER = {1},
     PAGES = {20--23},
      ISSN = {0022-314X,1096-1658},
   MRCLASS = {11B65},
  MRNUMBER = {1200804},
MRREVIEWER = {Andrew\ Granville},
       DOI = {10.1006/jnth.1993.1002},
       URL = {https://doi.org/10.1006/jnth.1993.1002},
}
\end{document}